\documentclass[letter,10pt]{article}
\usepackage{amssymb}
\usepackage{amsthm}
\usepackage{amsmath}
\usepackage{mathrsfs}
\usepackage{verbatim}
\usepackage{enumerate}

\newcommand{\R}{\mathbb{R}}
\usepackage{cite}
\usepackage{hyperref}
\usepackage{changes}
\usepackage{color}

\begin{document}
\title{A general condition for Monge solutions in the multi-marginal optimal transport problem\footnote{Y.-H.K. is supported in part by 
Natural Sciences and Engineering
Research Council of Canada (NSERC) Discovery Grants 371642-09 as well as Alfred P. Sloan research fellowship.  B.P. is pleased to acknowledge the support of a University of Alberta start-up grant and National Sciences and Engineering Research Council of Canada Discovery Grant number 412779-2012.
Part of this work was done while Y.-H.K was visiting University Paris-Est Cr\'eteil (UPEC) and Korea Advanced Institute of Science and Technology (KAIST) and he thanks for their hospitality. 
}}

\author{Young-Heon Kim\footnote{Department of Mathematics, University of British Columbia, Vancouver BC Canada V6T 1Z2 yhkim@math.ubc.ca} and Brendan Pass\footnote{Department of Mathematical and Statistical Sciences, 632 CAB, University of Alberta, Edmonton, Alberta, Canada, T6G 2G1 pass@ualberta.ca.}}
\maketitle
\begin{abstract}
We develop a general condition on the cost function which is sufficient to imply Monge solution and uniqueness results in the multi-marginal optimal transport problem.  This result unifies and generalizes several results in the rather fragmented literature on multi-marginal problems.  We also provide a systematic way to generate new examples from old ones.

\end{abstract}

\section{Introduction}

In this paper, we establish a general Monge solution and uniqueness result for the multi-marginal Monge-Kantorovich problem, under a natural analogue of the twist condition.  We call this condition \emph{twist on splitting sets}.

Given compactly supported  Borel probability measures $\mu_1,...\mu_m$ on smooth manifolds $M_1, M_2,...,M_m$, respectively, and a continuous cost function $c: M_1 \times M_2 \times,...,\times M_m \rightarrow \mathbb{R}$, the multi-marginal optimal transport problem is to minimize 
\begin{equation}\label{MK}
\int_{M_1 \times ...\times M_m} c(x_1,x_2,...,x_m)d\gamma,
\end{equation}
among probability measures $\gamma$ on $M_1 \times....\times M_m$ which project to the $\mu_i$.  When an optimal measure $\gamma$ is concentrated on the graph $\{(x,T(x))\}$ of a function $T:M_1 \rightarrow M_2 \times....\times M_m$, it is said to induce a \emph{Monge solution}.   When $m=2$, \eqref{MK} reduces to the classical Monge-Kantorovich problem, which remains a very active area with a wide variety of applications (see \cite{V2} for a comprehensive review).  Recently, applications for the $m \geq 3$ case have arisen in such diverse areas as matching in economics \cite{CE}\cite{CMN}, electronic correlations in physics \cite{CFK}\cite{bdpgg}, monotonicity relationships among vector fields \cite{GG}\cite{GhM}\cite{GhMa} and model free pricing of derivatives in finance \cite{ght}\cite{bhlp}\cite{hpt}.

Under reasonable conditions on the cost and marginals, existence of an optimal measure $\gamma$ is not hard to show.  Two natural open questions are: ``when is the optimal measure $\gamma$ unique?" and ``when does the optimal measure induce a Monge solution?"

In the $m=2$ case, the well known twist condition, dictating that the mapping $x_2 \mapsto D_{x_1}c(x_1,x_2)$ is injective for fixed $x_1$, ensures the uniqueness and Monge structure of the optimal $\gamma$ \cite{G}\cite{GM}\cite{lev}\cite{Caf}.  For larger $m$, these questions are still largely open.  Examples of special cost functions for which the optimal measure has this structure are known \cite{GS}\cite{H}\cite{C}\cite{P9}\cite{KP}, as well as several examples for which uniqueness and Monge solutions fail \cite{P}\cite{CN}.  There are also strong differential conditions on the cost which are known to imply Monge solutions and uniqueness \cite{P1}; however, these conditions are not sharp, as some of the positive examples do not satisfy them.  What seems to be missing is an analogue of the twist condition; that is, a general condition implying Monge solution and uniqueness results, which unifies the scattered, previously established results.

In this paper we propose such a condition on the cost, which we call \emph{twist on  $c$-splitting} sets  (or simply, {\em twist on splitting sets}), and show that it is indeed sufficient for Monge solutions and uniqueness.  We require the mapping $(x_2,...,x_m) \mapsto D_{x_1}c(x_1,x_2,...,x_m)$ to be injective along certain subsets, which we call splitting sets  (see Definition \ref{splitset} below); splitting sets, roughly speaking, are multi-marginal analogues of $c$-super differentials of $c$-concave functions.

We also consider a natural extension  %(for the first time, to the best of our knowledge)
of  the concept of $c$-cyclical monotonicity to multi-marginal problems.   As we show, any splitting set is automatically $c$-cyclically monotone.  The converse, when $m=2$, is a well known theorem of R\"uschendorf \cite{Ruschendorf}; whether the converse holds for $m \geq 3$ remains an interesting open question.  As an immediate corollary, we obtain Monge solution and uniqueness results whenever $c$ is twisted on $c$-cyclically monotone sets, which in practice may be more direct to check for a given cost than twistedness on splitting sets.

An important conceptual contribution of this paper is that it unifies and extends known Monge solution results for multi-marginal problems.  For example, Monge solution and uniqueness results for a class of costs called matching costs (due to their application in economics) was established in \cite{P9}.  These costs 
%{\red For example, the costs of the matching form in \cite{P9}} 
may not satisfy the differential conditions in \cite{P1}; in turn, there are costs satisfying the differential conditions which are not of matching form.  However,  both the differential conditions and the matching structure imply twist on splitting sets; indeed, we show that the conditions imposed in \cite{P1} are in fact sufficient (but not necessary) differential conditions for twist on splitting sets.  For cost functions of the form in \cite{P9}, we show twist in $c$-monotonicity is  satisfied as long as the $c_i$ are twisted, and therefore other conditions on the derivatives of the $c_i$, required for the argument in \cite{P9}, are not needed here. 
Indeed, we are able to extend, in a systematic way,  this type of example to a more general class, namely, costs defined as infimums of functions of less variables (see Section~\ref{S:infimal}).

 %Therefore, our main result here encompases and unifies several known Monge solution results for multi-marginal problems. 

Note that we work with semi-concave (not necessarily smooth) cost functions here.  This makes some of the definitions and proofs slightly more complicated and less elegant looking; on the other hand, we require  the semi-concave framework to handle natural examples where the cost is not everywhere smooth (as in Section~\ref{S:infimal}).  %,such as compact {\red Riemannian} manifolds or costs defined as infimums of functions of less variables (see subsections 4.3 and 4.4).
%\marginpar{July1}

In the next section, we introduce the key conditions we will use in this paper.  In the third section we state and prove our main theorem, while the final two sections are reserved for two key types of examples.  It is in these final sections that we show the results in \cite{P1}\cite{P9}\cite{GS}\cite{H}\cite{KP} fit into our framework.

\section{Preliminaries}
We now formulate the main concepts used in the paper.

\newtheorem{splitset}{Definition}[section]
\begin{splitset}\label{splitset}
A set $S \subseteq M_1 \times M_2 \times...\times M_m$ is a $c$-splitting set if there exists Borel functions $u_i:M_i \rightarrow \mathbb{R}$ such that for all $(x_1,x_2,...,x_m)$

\begin{equation}\label{eq:u i}
\sum_{i=1}^m u_i(x_i) \leq c(x_1,x_2,...,x_m)
\end{equation}
with equality whenever $(x_1,x_2,...,x_m) \in S$.  We will call the $u_i$ $c$-splitting functions  for $S$.
\end{splitset}

\newtheorem{cmonodef}[splitset]{Definition}
\begin{cmonodef}\label{cmonodef}
A set $S \subseteq M_1 \times M_2 \times...\times M_m$ is $c$-cyclically monotone % \sout{{\red  with respect to $x_1$}} 
  if for any finite subset $\{(x^1_1,x^1_2,...x^1_m),(x^2_1,x^2_2,...x^2_m),....(x^N_1,x^N_2,...x^N_m) \} \subseteq S$ and any $m$ permutations $\sigma_1,\sigma_2,...\sigma_m$ on $N$ letters, we have
\begin{equation*}
\sum_{i=1}^N c(x^i_1,x^i_2,...x^i_m) \leq \sum_{i=1}^N c(x^{\sigma_1(i)}_1,x^{\sigma_2(i)}_2,...x^{\sigma_m(i)}_m).
\end{equation*}
%\sout{{\red We can similarly define $c$-cyclical monotonicity with respect to $x^i$. We say a set in $M_1 \times \cdots \times M_m$ is $c$-cyclically monotone, if it is $c$-cyclically monotone with respect to each $x^i$, $i=1, ..., m$. }}
\end{cmonodef}
Note that, by considering the permutations $\sigma_i \circ \sigma_1^{-1}$, we can always take $\sigma_1=Id$ (or $\sigma_j =Id$, for any other fixed $j$) in the above definition.
%  \marginpar{\blue{I think $c$-cyclical monotonicity with respect to $x_1$ implies $c$-cyclical monotonicity with respect to $x_i$ (by composing with inverses of the permutations) so I suggest we simply define $c$-cyclical monotonicity as this.}}
The following result relates these concepts to optimal measures $\gamma$ in (\ref{MK}).
\newtheorem{cmonsupport}[splitset]{Proposition}
\begin{cmonsupport}\label{cmonsupport}
A probability measure $\gamma$ on $M_1 \times ...\times M_m$ is optimal in (\ref{MK}) for its marginals if and only if its support is a $c$-splitting set.  Any $c$-splitting set is $c$-cyclically monotone.

\end{cmonsupport}
\begin{proof}
The equivalence of the optimality of $\gamma$ and the splitting set property of its support follows easily from a classical duality theorem of Kellerer \cite{K}.  %\marginpar{Better to give additional reference for this.} 
We now prove that any $c$-splitting set $S$ is $c$-cyclically monotone.
Let $(u_1,u_2,...,u_m)$ be $c$-splitting functions for $S$.  Then, for any 
\begin{equation*}
\{(x^1_1,x^1_2,...,x^1_m),(x^2_1,x^2_2,...,x^2_m),....(x^N_1,x^N_2,...,x^N_m) \} \subseteq S
\end{equation*}
it is clear that we have 

\begin{equation}\label{splitequal}
\sum_{i=1}^N c(x^i_1,x^i_2,...x^i_m) = \sum_{i=1}^N\sum_{j=1}^m u_j(x_j^i).
\end{equation}

On the other hand, by the definition of splitting functions, we have, for any permutations $\sigma_2,\sigma_3,...,\sigma_m$ (and setting $\sigma_1 =id$)

\begin{equation}\label{permsplitinequal}
\sum_{i=1}^N\sum_{j=1}^m u_j(x_j^i)= \sum_{i=1}^N\sum_{j=1}^m u_j(x_j^{\sigma_j(i)}) \leq \sum_{i=1}^N c(x^i_1,x^{\sigma_2(i)}_2,...x^{\sigma_m(i)}_m).
\end{equation}
%{\sout{\red The $c$-cyclical monotonicity property with respect to $x^1$ follows easily from \eqref{splitequal} and \eqref{permsplitinequal}. Similarly obtained are $c$-cyclical monotonicity with respect to $x^i$, for other $i$'s. }}
\end{proof}

%\marginpar{omitted the two definitions here}
%In this paper, we will  mostly be interested in a related property, defined on the lower dimensional product space $M_2 \times...\times M_m$.
%\newtheorem{splitset2}[splitset]{Definition}
%\begin{splitset2}\label{splitset2}
%Fix $x_1 \in M_1$.  We call a set $S  \subseteq M_2 \times...\times M_m$ a $c$-splitting set at $x_1$ if it is a $c_{x_1}$-splitting set for the cost function $c_{x_1}(\cdot):=c(x_1,\cdot)$ on $M_2 \times...\times M_m$. 
%\end{splitset2}
%Note that the preceding definition is equivalent to the set $S_{x_1}:=\{(x_1,x_2,...,x_m): (x_1,x_2,...,x_m) \in S\}$ being a $c$-splitting set.
%
%We also have use for a similar notion, related to $c$-cyclical monotonicity.
%\newtheorem{cmonodef2}[splitset]{Definition}
%\begin{cmonodef2}
%Fix $x_1 \in M_1$.  We will say that a subset $S \subseteq M_1 \times M_2 \times...\times M_m$ is cyclically $c$-monotone at $x_1$ or \emph{$x_1$-$c$-monotone} if it is $c_{x_1}$-cyclical monotonone for the reduced cost function  $c_{x_1}(\cdot):=c(x_1,\cdot)$ on $M_2 \times...\times M_m$. In this case, in our convention, the permutations are for $i=3, ..., m$. 
%\end{cmonodef2}
%The notion $c$-cyclical monotonicity at $x_1$ is equivalent to $c$-cyclically montononicity of the set $S_{x_1}:=\{(x_1,x_2,...,x_m): (x_1,x_2,...,x_m) \in S\}$.

\newtheorem{tss}[splitset]{Definition}
\begin{tss}\label{D:twist}
Let $c$ be a continuous, semi-concave cost function.  We say $c$ is {\em twisted on $c$-splitting sets} (respectively, {\em twisted in $c$-cyclical monotonicity}) %{\sout if,}
 whenever for each fixed $x_1 \in M_1$ and $c$-splitting set (respectively, $c$-cyclically monotone set) $S \subseteq \{ x_1\} \times M_2 \times \cdots \times M_m$, the map
\begin{equation*}
(x_2,...,x_m) \mapsto D_{x_1}c(x_1,x_2,...,x_m)
\end{equation*} 
is injective on the subset of $S$ where $c$ is differentiable with respect to $x_1$ (i.e, the subset where $D_{x_1}c(x_1,x_2,...,x_m)$ exists).

\end{tss}

%\newtheorem{tcmono}[splitset]{Definition}
%\begin{tcmono}
%Let $c$ be a continuous, semi-concave cost function.  We say $c$ is twisted in $c$-cyclical monotonicity, {\sout if,} {\red whenever for each fixed $x_1 \in M_1$,  $S\subseteq \{ x_1\} \times M_2 \times \cdots \times M_m$ is a $c$-cyclically montone set with respect to $x_1$, the map}
%
%\begin{equation*}
%(x_2,...,x_m) \mapsto D_{x_1}c(x_1,x_2,...,x_m)
%\end{equation*} 
%is injective on the subset of $S$ where $c$ is differentiable with respect to $x_1$ .
%\end{tcmono}
  %\marginpar{combined the two definitions of two twistedness. }

When $m=2$, any set $S\subseteq \{x_1\} \times M_2$ is trivially  both $c$-cyclically monotone and a $c$-splitting set, so both twist on splitting sets and twist in $c$-cyclical monotonicity reduce to the standard twist condition.  For higher $m$, twist in $c$-cyclical monotonicity clearly implies twist on splitting sets, by Proposition \ref{cmonsupport}.  When $m=3$, a set $S\subseteq \{ x_1\} \times M_2 \times M_3$ is $c$-cyclical monotone  if and only if it is splitting set, by R\"uschendorf's theorem, and so twist in $c$-cyclical monotonicity and twist on splitting sets are equivalent.  We do not know whether this equivalence holds for larger $m$.% is an interesting open question.

\section{Monge solution and uniqueness}

We are now ready to state and prove the main  result.

\newtheorem{main}{Theorem}[section]\label{T:main}
\begin{main}
Assume $c$ is  twisted on splitting sets and the measure $\mu_1$ is absolutely continuous with respect to local coordinates.  Then the solution $\gamma$ in \eqref{MK} induces a Monge solution and is unique.
\end{main}
\begin{proof}

We first prove the Monge solution assertion. The key observation is that the twist on splitting sets condition is enough to 
%{\sout carryout a standard argument  (e.g. \cite{GM}) in the two marginal case, also in  the multimarginal case.} 
 extend a standard argument from the two marginal case (found in, for example, \cite{GM}) to the multi-marginal case.
By Proposition~\ref{cmonsupport}, there exist splitting functions $(u_1,u_2,...,u_m)$  for $spt(\gamma)$, and it is well known that they can be taken to be \emph{$c$-conjugate} \cite{K}\cite{GS}\cite{P1}; that is, for each $i$,
\begin{equation}\label{dual}
u_i(x_i) = \inf_{x_j, j\neq i}c(x_1,x_2,...,x_m) -\sum_{j \neq i}^mu_i(x_j).
\end{equation}
 %\sout{ Choose a $c$-conjugate solution $(u_1,u_2,...,u_m)$ to the dual problem {\red  given in Proposition~\ref{cmonsupport}.}} 
  In particular, as an infimum of semi-concave functions, $u_1$ is itself semi-concave and therefore differentiable almost everywhere with respect to local coordinates (and hence $\mu_1$ almost everywhere by absolute continuity).

Fix $x_1 \in spt(\mu_1)$ where $u_1$ is differentiable.  We must show that there exists a unique $(x_2,...x_m)$ such that $(x_1,x_2,...x_m) \in spt(\gamma)$: that is, the set 
\begin{equation*}
S= spt(\gamma) \cap \big[\{ x_1\} \times M_2 \times \cdots \times M_m\big]
%\\{(x_2,...x_m):(x_1,x_2,...x_m) \in spt(\gamma) \}
\end{equation*}
is a singleton.  Non-emptiness of $S$ follows immediately, as the support of $\gamma$ must project to the support of $\mu_1$.  Note that this set is a splitting set at $x_1$, as we have $\sum_{i=1}u_i(x_i) =c(x_1,...,x_m)$ on the support of $\gamma$ {from Proposition~\ref{cmonsupport}.}  By \eqref{dual}, then, for each $(x_1, x_2,...,x_m) \in S$, we must have

\begin{equation*}
\partial_{x_1} c(x_1,...,x_m) \subseteq \partial u_1(x_1) =\{Du_1(x_1)\}
\end{equation*}
where $\partial_{x_1} c(x_1,...,x_m)$ denotes the superdifferential of $c$ with respect to $x_1$; note that the last equality follows by the differentiability of $u_1$ at $x_1$.  It follows that $c$ is differentiable with respect to $x_1$ at $(x_1,...,x_m)$ and we have

\begin{equation*}
Du_1(x_1) = D_{x_1}c(x_1,x_2,...x_m).
\end{equation*}
As $(x_2,...x_m) \mapsto D_{x_1}c(x_1,x_2,...x_m)$ is injective on $S$ by the twist on splitting sets condition, this immediately implies that $S$ must be a singleton.

This shows that every solution to $\gamma$ is concentrated on a graph over the first variable.  Uniqueness follows by a standard argument; as the functional $\eqref{MK}$ is linear, the convex interpolant $\frac{1}{2}(\gamma + \hat \gamma)$ of any two solutions must also be a solution.  However, if $\gamma$ and $\hat \gamma$ are concentrated on graphs $T$ and $\hat T$, respectively, then  $\frac{1}{2}(\gamma + \hat \gamma)$ is concentrated on the union of the graphs of $T$ and $\hat T$; this set itself cannot be a graph unless $T=\hat T$ almost everywhere.  Uniqueness of the optimal measure $\gamma =(ID,T)|_{\#}\mu_1$ follows immediately.
\end{proof}
The following result now follows easily from Theorem~\ref{T:main} and Proposition~\ref{cmonsupport}.
\newtheorem{maincor}[main]{Corollary}
\begin{maincor}
Assume $c$ satisfies the twist in $c$-monotonicity condition and $\mu_1$ is absolutely continuous with respect to local coordinates.  Then the solution $\gamma$ in \eqref{MK} induces a Monge solution and is unique.
\end{maincor}

%\section{Examples}%Here we exhibit several example classes of cost that satisfy the twist on splitting sets condition.

\section{Differential conditions}
We now exhibit several example classes of cost functions that satisfy the twist on splitting sets condition.
First, in this section, we show that the differential conditions in \cite{P1} imply twist on splitting sets.  Let us  recall those conditions:

Let $M_i \subseteq \R^n$, $i=1, ..., m$. We will assume throughout this section that $c$ is $(1,m)$-twisted; that is, the map

\begin{equation*}
x_m \mapsto D_{x_1}c(x_1,x_2,...x_m)
\end{equation*}
is injective for fixed $x_1,...x_{m-1}$. We will also assume that $c$ is $(1,m)$-non-degenerate; that is, the matrix 

\begin{equation*}
D^2_{x_1x_m}c =(\frac{\partial^2 c}{\partial x^i_1 \partial x^j_m})_{ij}
\end{equation*}
is everywhere non-degenerate.

The most restrictive condition in \cite{P1} is based on the following tensor.
\newtheorem{tensor}{Definition}[section]

\begin{tensor}
Suppose $c$ is $(1,m)$-non-degenerate.  Let $\vec{y}=(y_1,y_2,...,y_m) \in M_1 \times M_2 \times...\times M_m$.  For each $i:=2,3,...,m-1$ choose a point $\vec{y}(i)=(y_1(i),y_2(i),...,y_m(i)) \in \overline{M_1} \times \overline{M_2} \times...\times \overline{M_m}$ such that $y_i(i)=y_i$.  Define the following bi-linear maps on $T_{y_2}M_2 \times T_{y_3}M_3 \times ...\times T_{y_{m-1}}M_{m-1}$:

\begin{equation*}
 S_{\vec{y}}=-\sum_{j=2}^{m-1} \sum_{\substack {i=2 \\ i \neq j}}^{m-1}D^2_{x_ix_j}c(\vec{y}) +\sum_{i,j=2}^{m-1}(D^2_{x_ix_m} c \, (D^{2}_{x_1x_m}c)^{-1}\, D^2_{x_1x_j}c)(\vec{y})
\end{equation*}

\begin{equation*}
 H_{\vec{y},\vec{y}(2),\vec{y}(3),...,\vec{y}(m-1)}=\sum_{i=2}^{m-1}({\rm Hess}_{x_i} c(\vec{y} (i))-{\rm Hess}_{x_i}c(\vec{y}))
\end{equation*}

\begin{eqnarray*}
T_{\vec{y},\vec{y}(2),\vec{y}(3),...,\vec{y}(m-1)}=S_{\vec{y}}+H_{\vec{y},\vec{y}(2),\vec{y}(3),...,\vec{y}(m-1)}
\end{eqnarray*}
\end{tensor}
The main condition required for Monge solutions in \cite{P1} is negative definiteness of the tensor $T$, for all choices of the $\vec{y},\vec{y}(2),\vec{y}(3),...,\vec{y}(m-1)$.

A geometric condition on the domains, defined in terms of the following set, is also required.

\newtheorem{domains}[tensor]{Definition}
\begin{domains}\label{domains}
 Let $x_1 \in M_1$ and $p_1 \in T^*_{x_1}M_1$.  We define $Y^c_{x_1,p_1} \subseteq M_2 \times M_3 \times ...\times M_{m-1}$ by
\begin{equation*}
 Y^c_{x_1,p_1} =\{(x_2,x_3,...,x_{m-1}) |\text{ } \exists\text{ } x_m \in M_m \textit{ s.t. }  D_{x_1}c(x_1,x_2,...,x_m) = p_1\}
\end{equation*}

\end{domains}

These conditions are discussed in more detail in \cite{P1}.  
Although they are fairly restrictive, several examples of cost functions satisfying these conditions are exhibited in \cite{P1}, including the Gangbo-Swiech \cite{GS} cost, $\sum_{ i\neq j}|x_i-x_j|^2$ on $\mathbb{R}^n$, and perturbations thereof, the cost function considered by Heinich \cite{H}, $h(\sum_{i=1}^m x_i)$ for a strictly concave $h:\mathbb{R}^n \rightarrow \mathbb{R}$, and three marginal functions of the form $x_1\cdot x_2 +x_2\cdot x_3 +g(x_1,x_3)$, whenever $D^2_{x_1x_3}g >0$.

Below, we prove that these conditions imply the twist on splitting sets condition.

\newtheorem{diffcond}[tensor]{Proposition}
\begin{diffcond}
(Sufficient differential conditions) 

Suppose that:
\begin{enumerate}
\item $c$ is $(1,m)$-non-degenerate. 
\item $c$ is $(1,m)$-twisted.
\item For all choices of $\vec{y}=(y_1,y_2,...,y_m) \in M_1 \times M_2 \times...\times M_m$ and of $\vec{y}(i)=(y_1(i),y_2(i),...,y_m(i)) \in \overline{M_1}\times \overline{M_2} \times...\times \overline{M_m}$ such that $y_i(i)=y_i$ for $i=2,...,m-1$, we have 

\begin{equation*}
T_{\vec{y},\vec{y}(2),\vec{y}(3),...,\vec{y}(m-1)} < 0. 
\end{equation*}
\item For all $x_1 \in M_1$ and $p_1 \in T^*_{x_1}M_1$, $Y_{x_1,p_1}^c$ is geodesically convex.
\end{enumerate}
Then $c$ is twisted on splitting sets.
\end{diffcond}

\begin{proof}
Fix $x_1 \in M_1$ and let $S \subseteq \{x_1\} \times M_2 \times ....\times M_m$ be a $c$-splitting set. Take points $(x_2,...,x_m)$ and $(\bar x_2,...,\bar x_m)$ such that $p_1:=D_{x_1}c(x_1,x_2...x_m) = D_{x_1}c(x_1,\bar x_2...\bar x_m)$.  It suffices to show that at most one of these two points can be in $M_1$.  Without loss of generality, assume   $(x_2,...,x_m) \in S$; we must show that $(\bar x_2,...,\bar x_m) \notin S$.
 
For $i=2,3,...,m-1$, choose geodesics joining $\gamma_i(t)$ $x_i=\gamma_i(0)$ and $\bar x_i=\gamma_i(1)$.  The geodesic convexity of $Y_{x_1,p_1}^c$ implies that for each $t$, there exists an $x_m(t) \in M_m$ such that $p_1=D_{x_1}c(x_1,\gamma_2(t)...\gamma_{m-1}(t),x_m(t))$.  Twistedness ensures the uniqueness of $x_m(t)$.  Note that $x_m(0) =x_m$, and $x_m(1) = \bar x_m$.

Now, as $S$ is a splitting set, we have splitting functions $u_i:M_i \rightarrow \mathbb{R}$ such that
\begin{equation*}
\sum_{i=2}^m u_i(x_i) \leq c(x_1,x_2,...,x_m),
\end{equation*}
with equality on $S$.  By a standard convexification trick and compactness of the $M_i$, we can assume that 
\begin{equation*}
u_i(x_i) =\min_{x_j,j\neq i} c(x_1,...x_m) -\sum_{j=2,j\neq i}u_j(x_j).
\end{equation*}
The functions $u_i$ are semi-concave and have superdifferentials everywhere.  Furthermore, for any $(x_2,...,x_m) \in S$, we have

\begin{equation*}
D_{x_i}c(x_1,x_2,....,x_m) \in \partial u_i(x_i)
\end{equation*}
for $i=2,3,...,m$, where $\partial u_i(x_i)$ denotes the super-differential of $u_i$.

Take a measurable selection of covectors $V_i(t) \in \partial u_i(\gamma_i(t))$, and set. 
\begin{equation*}
f(t):=\sum_{i=2}^{m-1}[V_i(t)-D_{x_i}c(x_1,\gamma_2(t),...,\gamma_{m-1}(t),x_m(t))]\langle\frac{d\gamma_i}{dt}\rangle
\end{equation*}
Now, note that we can take $V_i(0) = D_{x_i}c(x_1,x_2,....,x_m)$, as $(x_2,....,x_m) \in S$, in which case $f(0) =0$.  Similarly, \emph{if} $(\bar x_2,...,\bar x_m) \in S$, we can choose $V_i(1) =D_{x_i}c(x_1,x_2,....,x_m)$, in which case $f(1) =0$.

However, the calculation in \cite{P1} (in the proof of Theorem 3.1) it is shown that under the conditions $1,2,3$ and $4$, $f(1) <f(0)$, for any selection of covectors $V_i(t) \in \partial u_i(\gamma_i(t))$.  This implies that we cannot have $(\bar x_2,...,\bar x_m) \in S$, completing the proof.
\end{proof} 

\section{Infimal convolution examples}\label{S:infimal} %of twisted in $c$-cyclical monotonicity costs}
In this section we consider a sort of infimal convolution of several cost functions; that is, cost functions defined by
 \begin{equation}\label{infcon}
c(x_1, ..., x_m)= c(X_1,X_2,...,X_k) = \min_{y \in Y} \sum_{j=1}^k c_j(X_j, y).
\end{equation}
%\begin{equation}\label{infcon}
%c(x_1,x_2,...,x_m) = \min_{y \in Y} \big[c_1(x_1,x_2,...,x_k,y) + c_2(x_{k+1},x_{k+2},...,x_m,y)\big].
%\end{equation}
Here, for notational convenience, we decomposed the $m$-tuple $(x_1, ..., x_m)$ into {$k$ smaller tuples} $(X_1, ..., X_k)=(x_1, ..., x_m)$ with $X_j = (x_{m_{j-1} + 1}, \cdots, x_{m_j})$, with $0=m_0< m_1 < m_2 < \cdots < m_k=m$; in particular, $X_1 = (x_1, ..., x_{m_1})$. 
We also assume that $Y$ is a smooth manifold without boundary, and we are implicitly assuming the existence of a minimizing $y$ for all $(x_1,x_2,...,x_m)$%\sout{: for example,  we can assume $Y$ is compact}
{(which holds, for example, whenever $Y$ is compact)}. We also assume that the functions $c_j$  are  semi-concave, so that $c$ is also semi-concave.

As a special case, {when each $X_j$ is a singleton, }we have

\begin{equation}\label{match}
c(x_1,...,x_m) := \min_{y \in Y} \sum_{i=1}^m c_i(x_i,y).
\end{equation}
These cost functions, called {\em matching costs},  have important applications in matching problems in economics \cite{CE}\cite{CMN}.  In \cite{P9}, one of the present authors proved a Monge solution and uniqueness result for costs of this form.  The argument was completely different than the one here, and required additional conditions on the $c_i$, including nondegeneracy of various matrices of mixed second order partials and uniqueness of the minimizing $y$.  In a recent preprint, we studied the special case when each $c_i$ is the distance squared on a Riemannian manifold \cite{KP}; in this case, the smoothness and non-degeneracy conditions required in \cite{P9} may fail, and the techniques developed there are closer to those used in this paper. 

Our main result in this direction is the following, which gives a systematic way to generate new multi-marginal cost functions which ensure the Monge solution  structure and uniqueness of the optimal measure.

\newtheorem{infs}{Theorem}[section]
\begin{infs}\label{infs}

%\begin{enumerate}
%\item 
%Suppose $c_1(x_1,x_2,...x_k,y)$ and $c_2(x_{k+1},x_{k+2},...x_m,y)$ are smooth funtions. 
% Assume that each $c_i$, $i=1, ..., k$, satisfies twist on splitting sets.
%Then the cost $c(x_1,x_2,...,x_m)$ defined by \eqref{infcon} satisfies  twist on splitting sets.
%\item %Suppose $c_1(x_1,x_2,...x_k,y)$ and $c_2(x_{k+1},x_{k+2},...x_m,y)$ are smooth funtions. 
 Assume $c_1$ satisfies twist in $c_1$-cyclical monotonicity
% with respect to $x_1$ 
 and $c_j$, $j=2, ..., k$, satisfies twist in $c_j$-cyclical monotonicity with respect to $y$, i.e., the map $X_j \in S \mapsto D_{y} c_j ( X_j, y)$ is injective along $c$-monotone subsets $S \subseteq M_{m_{j-1}+1}  \times \cdots M_{m_j}\times \{y\}$.
 %{\red i.e in the Definition~\ref{D:twist}, we put the map $X_j \in M_{m_{j-1}+1}  \times \cdots M_{m_j}\times Y \mapsto D_{y} c_j ( X_j, y)$ instead,  keeping the others  the same.} 
 Then the cost $c(x_1,x_2,...,x_m)$ defined by \eqref{infcon} satisfies  twist in $c$-cyclical monotonicity. %with respect to $x_1$.

%\end{enumerate}
\end{infs}

%\begin{infs}\label{th:infs}
%%Suppose $c_1(x_1,x_2,...x_k,y)$ and $c_2(x_{k+1},x_{k+2},...x_m,y)$ are smooth funtions. 
% Assume $c_1$ satisfies twist in $c_1$-cyclical monotonicity with respect to $x_1$ and $c_i$, $i=2, ..., k$, satisfies twist in $c_i$-cyclical monotonicity with respect to $y$.  Then the cost $c(x_1,x_2,...,x_m)$ defined by \eqref{infcon} satisfies  twist in $c$-cyclical monotonicity with respect to $x_1$.
%\end{infs}
\newtheorem{rem}[infs]{Remark}
%\begin{remark}

\begin{rem}
In this theorem, in fact, a slightly stronger result holds, namely, one can replace twist in  $c$-cyclical monotonicity in the %\sout{statement} 
conclusion, with twist in $c$-monotonicity of order two, which is defined exactly %\sout{the same} 
 as in Definition \ref{cmonodef}, except the number $N$ there is fixed to be $N=2$.  This will be obvious by examining the proof.
%\end{remark}
\end{rem}

 The proof of this result is based on the same essential idea as our argument in \cite{KP} and is divided into several Lemmas.  The first two of  these show that a $c$-cyclically monotone set $S$ projects,  in a certain sense, to $c_j$-cyclically monotone sets. 

%%\sout{
% For the proof, the essential idea is the same as in \cite{KP}. We divide the proof of this result into several Lemmas.
% First we show that the cost-cyclical monotonicity survives restriction to smaller number of variables.
\newtheorem{c_1mono}[infs]{Lemma}
\begin{c_1mono}\label{c_1mono}
Suppose the set $S \subseteq M_1 \times ....\times M_m$ is $c$-cyclically-monotone.  Use the notation given in the beginning of this section.  Then the set
\begin{eqnarray*}
\bar S &:=& \{(X_1,y): \exists (X_2, ...., X_k) \text{ such that } (X_1,..., X_k) \in S\text{ and } \\
&& y \in {\rm argmin}  \big[\sum_{j=1}^k c_j (X_j, y) \big]\}\\
&& \subseteq M_1 \times...\times M_{m_1} \times Y.
\end{eqnarray*}
is $c_1$-cyclically-monotone. 
\end{c_1mono}
%\newtheorem{c_1mono}[infs]{Lemma}
%\begin{c_1mono}
%Suppose the set $S \subseteq M_2 \times ....\times M_m$ is $x_1$-$c$-cyclically-monotone.  Then the set
%\begin{eqnarray*}
%\bar S &:=& \{(x_2,...x_k,y): \exists (x_{k+1}, ....x_m) \text{ such that } (x_2,...x_m) \in S\text{ and } \\
%&& y \in {\rm argmin}  \big[\sum_{j=1}^k c_j (X_j, y) c_1(x_1,x_2,...x_k,y) + c_2(x_{k+1},x_{k+2},...x_m,y)\big]\}\\
%&& \subseteq M_2 \times...\times M_k \times Y.
%\end{eqnarray*}
%is $x_1$-$c_1$-cyclically-monotone. 
%\end{c_1mono}

\begin{proof}
The proof is straightforward 
and we include it here for the reader's convenience. 
Given $$(X^i_1 , y^i)= (x^i_1, x_2^i, ...x^i_{m_1},y^i) \in \bar S,$$ for $i=1,2...,l$, permutations $\sigma_j$ on $l$ letters for $j=2,3,...,m_1$ and a permutation $\eta$ on $l$ letters (corresponding to the $y$ argument in $c_1$),  we need to show

\begin{equation*}
\sum_{i=1}^l c_1(x^i_1, x^i_2,...x^i_{m_1},y^i) \leq \sum_{i=1}^l c_1(x^i_1, x^{\sigma_2(i)}_2,...x^{\sigma_{m_1}(i)}_{m_1},y^{\eta(i)}).
\end{equation*} 
Now, for each $i$ we can choose $(X^i_2, ....X^i_k)$ so that  $(X_1^i, X^i_2,...,X^i_k) \in S$ and 
\begin{equation*} 
y^i \in  {\rm argmin}  \big[\sum_{j=1}^k c_j (X^i_j, y^i) \big].
%[c_1(x^i_1,x^i_2,...x^i_k,y) + c_2(x^i_{k+1},x^i_{k+2},...x^i_m,y)].
\end{equation*}
Now, for $j=m_1+1,..., m$, choose $\sigma_j=\eta$. Then, from $c$-cyclical monotonicity, we have
\begin{eqnarray*}
& & \sum_{i=1}^l c_1(x^i_1, x_2^i,...x^i_{m_1},y^i)+c_2(x^i_{m_1+1},...x^i_{m_2},y^i) + \cdots + c_k(x^i_{m_{k-1}+1},.., x^i_{m},y^i)\\
& =& \sum_{i=1}^l c(x^i_1, x^i_2,...x^i_m)\\
&\leq & \sum_{i=1}^l c(x^i_1, x^{\sigma_2(i)}_2,...x^{\sigma_{m_1}(i)}_{m_1},x^{\eta(i)}_{m_1+1},...x^{\eta(i)}_m)\\
&\leq & 
\sum_{i=1}^l c_1(x^i_1, x_2^{\sigma_2(i)},...x^{\sigma(i)}_{m_1},%\text{\sout{$y^{\sigma_{m_1}(i)}$}} 
{ y^{\eta(i)}})+c_2(x^{\eta(i)}_{m_1+1},...x^{\eta(i)}_{m_2},y^{\eta(i)}) \\
&  &  \ \ \ \ \ \ \ \  + \cdots + c_k(x^{\eta(i)}_{m_{k-1}+1},.., x^{\eta(i)}_{m},y^{\eta(i)}).
%& & 
%\sum_{i=1}^l c_1(x_1, x^{\sigma_2(i)}_2,...x^{\sigma_k(i)}_k,y^{\eta(i)})+c_2( x^{\eta(i)}_{k+1},...,x^{\eta(i)}_m, y^{\eta(i)})
\end{eqnarray*}
Noting that 
\begin{align*}
& \sum_{i=1}^l c_2(x^i_{m_1+1},...x^i_{m_2},y^i) + \cdots + c_k(x^i_{m_{k-1}+1},.., x^i_{m},y^i) \\ &=  \sum_{i=1}^l 
c_2(x^{\eta(i)}_{m_1+1},...x^{\eta(i)}_{m_2},y^{\eta(i)}) 
+ \cdots + c_k(x^{\eta(i)}_{m_{k-1}+1},.., x^{\eta(i)}_{m},y^{\eta(i)}),
\end{align*}
we have
% \sum_{i=1}^l c_2( x^{\eta(i)}_{k+1},...,x^{\eta(i)}_m, y^{\eta(i)})$, the above becomes
\begin{equation*}
\sum_{i=1}^l c_1(x_1, x^i_2,...x^i_k,y^i) \leq \sum_{i=1}^l c_1(x_1, x^{\sigma_2(i)}_2,...x^{\sigma_k(i)}_k,y^{\eta(i)}),
\end{equation*}
which completes the proof.

\end{proof}

 Similarly, we have 
\newtheorem{c_2mono}[infs]{Lemma}
\begin{c_1mono}\label{c_2mono}
Suppose the set $S \subseteq M_1 \times ....\times M_m$ is $c$-cyclically-monotone.  
Fix $j$. 
Then the set
\begin{eqnarray*}
\bar S &:=& \{(X_j,y): \exists X_i \text{ for $i\ne j$}  \text{ such that } (X_1,..., X_k) \in S\text{ and } \\
& &y \in {\rm argmin}  \big[\sum_{i=1}^k c_i (X_i, y) \big]\}\\
& &\subseteq M_{m_{j-1}+1}  \times...\times M_{m_j} \times Y.
\end{eqnarray*}
is $c_j$-cyclically-monotone.% in the sense that 
%\begin{equation*}
%\sum_{i=1}^l c_j( x^i_{m_{j-1}+1},...x^i_{m_j},y^i) \leq \sum_{i=1}^l c_j( x^{\sigma_{m_{j-1}+1}(i)}_{m_{j-1}+1},...x^{\sigma_{m_j}(i)}_{m_j},y^i).
%\end{equation*} 
\end{c_1mono}

%
%\newtheorem{c_3mono}[infs]{Lemma}
%\begin{c_2mono}\label{c_2mono}
%Suppose the set $S \subseteq M_2 \times ....\times M_m$ is $x_1$-$c$-cyclically-monotone.  Fix $X_1=(x_1,x_2,...,x_k)$ and $y$.  Then the set
%\begin{eqnarray*}
%\tilde S& =& \{(x_{k+1},...x_m): \text{ such that } (x_2,...x_m) \in S\text{ and }\\
%&& y \in {\rm argmin}  \big[c_1(x_1,x_2,...x_k,y) + c_2(x_{k+1},x_{k+2},...x_m,y)\big]\} \\
%&\subset & M_{k+1} \times...\times M_m.
%\end{eqnarray*}
%is $y\text{-}c_2\text{-}$cyclically-monotone. 
%\end{c_2mono}
%
\begin{proof}
The proof is very similar to the proof of the preceding lemma and is skipped. 
%Given $(x^i_{k+1},...x^i_m) \in S$, for $i=1,2.,,,l$, and permutations $\sigma_j$ on $l$ letters for $j=k+1,...m$ ,  we need to show
%
%\begin{equation*}
%\sum_{i=1}^l c_2( x^i_{k+1},...x^i_m,y) \leq \sum_{i=1}^l c_2( x^{\sigma_{k+1}(i)}_{k+1},...x^{\sigma_m(i)},y).
%\end{equation*} 
%Now, by monotonicity, we have
%\begin{eqnarray*}
%& & \sum_{i=1}^l c_1(x_1, x_2,...x_k,y)+c_2( x^i_{k+1},...x^i_m,y) \\&=& \sum_{i=1}^l c(x_1, x_2,...x_k,x^i_{k+1},...,x^i_m)\\
%& \leq &  \sum_{i=1}^l c(x_1, x_2,...,x_k,x^{\sigma_{k+1}(i)}_{k+1},...,x^{\sigma_{m}(i)}_m)\\
%& = & \sum_{i=1}^l c_1(x_1, x_2,...,x_k,y)+ c_2( x^{\sigma_{k+1}(i)}_{k+1},...x^{\sigma_{k+1}(i)}_m,y)
%\end{eqnarray*}
%
%Canceling the  $c_1(x_1, x_2,...,x_k,y)$ terms above yields the desired result.
\end{proof}

\newtheorem{unique_y}[infs]{Lemma}
\begin{unique_y}\label{unique_y}
Fix $x_1 \in M_1$ and suppose the set $S \subseteq \{x_1\}\times M_2 \times ....\times M_m$ is $c$-cyclically-monotone and $c_1$ satisfies the twist in cyclical monotonicity condition.  Choose $(x_2,...x_m)$ and $(\bar x_2,...,\bar x_m)$ in $S$.  Let 
\begin{equation*}
y \in  {\rm argmin} \sum_{j=1}^k c_j (X_j, y)
%y \in  {\rm argmin} [c_1(x_1,x_2,...x_k,y) + c_2(x_{k+1},x_{k+2},...x_m,y)]
\end{equation*}
and 
 \begin{equation*}
 \bar y \in  {\rm argmin} \sum_{j=1}^k c_j (\bar X_j, y),
% \bar y \in  {\rm argmin} [c_1(\bar x_1,\bar x_2,...\bar x_k,y) + c_2(\bar x_{k+1},\bar x_{k+2},...\bar x_m,y)].
\end{equation*}  
where $(X_1, ..., X_k) = (x_1, x_2, ..., x_m)$ and $(\bar X_1, ..., \bar X_k) = (x_1, \bar x_2, ..., \bar x_m)$. 
If 
\begin{equation*}
D_{x_1}c(x_1, x_2,...,x_m) = D_{x_1}c(x_1, \bar x_2,...,\bar x_m),
\end{equation*}
then we must have $y =\bar y $, and $x_i = \bar x_i$, for $i=2,3,...,m_1$.

\end{unique_y}
\begin{proof}
Note that existence of the derivative $D_{x_1}c_1(x_1,x_2,....x_{m_1}, y)$ 
(respectively, $D_{x_1}c_1(x_1,\bar x_2,....\bar x_{m_1}, \bar y)$) implies the existence of $D_{x_1}c(x_1,x_2,....x_m)$ (respectively $D_{x_1}c(x_1,\bar x_2,...,\bar x_m)$) by standard arguments, and we have 
\begin{align*}
&D_{x_1}c_1(x_1,x_2,....x_{m_1}, y) = D_{x_1}c(x_1,x_2,....x_m) \\
&= D_{x_1}c(x_1,\bar x_2,...,\bar x_m) =  D_{x_1}c_1(x_1,\bar x_2,....\bar x_{m_1}, \bar y).
\end{align*}
The result now follows, since $c_1$ is twisted in  $c_1$-cyclical monotonicity and the projection of $S$ to $M_1 \times \cdots \times M_{m_1}$ is $c_1$-cyclically monotone from Lemma \ref{c_1mono}. 
\end{proof}

%We are now ready to prove Theorem \ref{th:infs}.

\begin{proof}[\bf Proof of Theorem \ref{infs}]
%{\red We give proof for the statement 1 and 2 simultaneously.} %The exactly same proof works for statement 1:  in the below, one can simply start with  a $c$-splitting set $S$ instead of $c$-cyclically monotone set.  But, since $c$-splitting property implies $c$-monotonicity, from then exactly the same proof works.} 
%\marginpar{{Done \blue Is the proof of 1 really the same? Is it obvious that versions of the preceding Lemmas hold with splitting sets in place of monotonicity?}}
 Fix $x_1 =\bar x_1 \in M_1$. 
Let $S$ be 
%a $c$-splitting set, respectively,  
$c$-cyclically-monotone in $ \{ x_1\} \times M_2 \times ....\times M_m$. %From now on we will only use the $c$-cyclical monotonicity of $S$; note that from Proposition~\ref{cmonsupport}, $c$-splitting property implies $c$-cyclical monotonicity.
Let $(x_1,...,x_m)$, $(\bar x_1,...,\bar x_m) \in S$ such that 
\begin{equation*}
D_{x_1}c(x_1,x_2...,x_m) = D_{x_1}c(\bar x_1,\bar x_2...,\bar x_m),
\end{equation*} 
we need to show $x_i =\bar x_i$ for all $i=2,3...,m$. (Here, the existence of these derivatives is part of the assumption.)
For $i=2,..., m_1$, this follows immediately from Lemma \ref{unique_y}.

To take care of the other $i$, let us
 use the notation 
\begin{align*}
&X_j = ( x_{m_{j-1}+1}, ...,  x_{m_j}); \\
& \bar X_j = (\bar x_{m_{j-1}+1}, ..., \bar x_{m_j}).
\end{align*}  
From the same Lemma \ref{unique_y}, we  obtain the existence of a $y$ such that  
\begin{equation}\label{eq:common y}
y \in  \Big[{\rm argmin} \sum_{j=1}^k c_j (X_j, y)\Big] \bigcap \Big[{\rm argmin} \sum_{j=1}^k c_j (\bar X_j, y)\Big].
%y \in  {\rm argmin} [c_1(x_1,x_2,...x_k,y) + c_2(x_{k+1},x_{k+2},...x_m,y)]
\end{equation}
%and 
% \begin{equation*}
% \bar y \in  {\rm argmin} \sum_{j=1}^k c_j (\bar X_j, y),
%% \bar y \in  {\rm argmin} [c_1(\bar x_1,\bar x_2,...\bar x_k,y) + c_2(\bar x_{k+1},\bar x_{k+2},...\bar x_m,y)].
%\end{equation*}  
%
%\begin{equation*}
%y \in {\rm argmin}_y [c_1(x_1,x_2,...,x_k,y) +c_2( x_{k+1},..., x_m, y)]
%\end{equation*}
%and  
%\begin{equation*}
%y \in {\rm argmin}_y \big[c_1(x_1,x_2,...,x_k,y) +c_2(\bar x_{k+1},...,\bar x_m, y)\big].
%\end{equation*}
We then obtain, by minimality of $\sum_{j=1}^k c_j (X_j, y) $ and $\sum_{j=1}^k c_j (\bar X_j, y)$
%$ c_1(x_1,x_2,...,x_k,y) +c_2(x_{k+1},...,x_m,y)$ 
at $y$,
\begin{align}\label{eq:critical y}
D_y  \sum_{j=1}^k c_j (X_j, y) =0,  \quad D_y  \sum_{j=1}^k c_j (\bar X_j, y) =0
\end{align}
Here, the differentiability of these derivatives follows from the semi-concavity of $c_j$'s together with the minimality at $y$: a semi-concave function $f$ should be differentiable at a minimum point. This last fact can be seen  easily by considering the superdifferential of the function $f$ (i.e. the subdifferential of $-f$), because, if the superdifferential $\partial f$ at a point $x_0$ has an element other than $0$, then $x_0$ cannot be a minimum point by the definition of superdifferential.  %\marginpar{Done. \blue{Is there a reference for the existence of derivatives of semi-concave functions at their minima?}}
%\begin{equation*}
%-D_yc_1(x_1,x_2,...,x_k,y) =D_yc_2(x_{k+1},...,x_m,y),
%\end{equation*}
%and similarly,
%\begin{equation*}
%-D_yc_1(x_1,x_2,...,x_k,y) =D_yc_2(\bar x_{k+1},...,\bar x_m,y).
%\end{equation*}

 Now, for fixed $l$ with $1\le l\le k $ let $X_l'=\bar X_l$ and $X_j' = X_j$ for $j\neq l$.  Similarly, set $\bar X_l'= X_l$ and $\bar X_j' = \bar X_j$ for $j\neq l$.  Note that by the definition of $c$, 

\begin{eqnarray*}
c(X_1,X_2'...,X_k') + c(X_1,\bar X_2'...,\bar X_k') &\leq& \sum_{j=1}^k c_j(X_j',y) +\sum_{j=1}^k c_j(\bar X_j',y) \\
&=&\sum_{j=1}^m c_j(X_j,y) +\sum_{j=1}^m c_j(\bar X_j,y)\\
&=&c(X_1,X_2...,X_k) + c(X_1,\bar X_2...,\bar X_k).
\end{eqnarray*}
On the other hand, by $c$-monotonicity, we must have 
\begin{equation*}
c(X_1,X_2...,X_k) + c(X_1,\bar X_2...,\bar X_k) \leq c(X_1,X_2'...,X_k') + c(X_1,\bar X_2'...,\bar X_k').
\end{equation*}
In light of the preceding series of inequalities, this implies that 

\begin{equation*}
c(X_1,X_2'...,X_k') + c(X_1,\bar X_2'...,\bar X_k') = \sum_{j=1}^k c_j(X_j',y) +\sum_{j=1}^k c_j(\bar X_j',y),
\end{equation*}
and $y \in {\rm argmin}\sum_{j=1}^k c_j(X_j',y)$, so that $\sum_{j=1}^m D_{y}c_j(X_j',y)=0$, or 
\begin{equation*}
\sum_{j\neq l}^k D_{y}c_j(X_j,y)=-D_{y}c_l(\bar X_l,y).
\end{equation*}
(Here again, the differentiability of  these functions follows from the semi-concavity of $c_j$'s together with the minimality at $y$. )
Now, from \eqref{eq:critical y},%as $y \in {\rm argmin}\sum_{j=1}^m c_j(x_j,y)$, we also have 
\begin{equation*}
\sum_{j\neq l}^k D_{y}c_j(X_j,y)=-D_{y}c_l( X_l,y).
\end{equation*}
We therefore conclude that 
\begin{equation}\label{eq:conclusion}
D_{y}c_l(  X_l,y) = D_{y}c_l( \bar X_l,y).
\end{equation}
 
%Therefore,
%\begin{equation*}
% D_yc_2(x_{k+1},...,x_m,y)=D_yc_2(\bar x_{k+1},...,\bar x_m,y).
%\end{equation*}
By Lemma \ref{c_2mono} and the twist in $c_l$-monotonicity with respect to $y$, we obtain $X_l =\bar X_l$. Since $l$ is arbitrary, this shows $x_i = \bar x_i$ for $1\le i \le m$, completing  the proof.

%{\red For the statement 2, noting that $c$-splitting property implies $c$-cyclical monotonicity (thus $c$-monotonicity), the only place to take care of is to have common $y$ as in \eqref{eq:common y}: in particular, we can  finish once having \eqref{eq:conclusion}  from the $c_l$-splitting property for each $l=1, ..., k$. To see the common $y$, let  %\begin{equation*}
%$y \in  {\rm argmin} \sum_{j=1}^k c_j (X_j, y)$
%and 
% $\bar y \in  {\rm argmin} \sum_{j=1}^k c_j (\bar X_j, y).$
%Now,  from the splitting property, there are functions $u_i: M_i \to \R$ as in Definition~\ref{splitset} such that 
%\begin{align*}
% \sum_j c_j (X_j, y)  &= c(x_1, ..., x_m) = \sum u_i (x_i)\\
%  \sum_j c_j (\bar X_j, \bar y)  &= c(\bar x_1, ..., \bar x_m) = \sum u_i (\bar x_i)\\
%\end{align*}
%From the property  \eqref{eq:u i},  
%\begin{align*}
% D_{x_1} c_1 (X_1, y)  &= D_{x_1}  c(x_1, ..., x_m) = D_{x_1} u_1 (x_1)\\
%  D_{x_1}  c_1(\bar X_1, \bar y)  &= D_{x_1} c(\bar x_1, ..., \bar x_m) = D_{x_1}  u_1 (\bar x_1)= D_{x_1} \\
%\end{align*}
%
% }
 % for $i=k+1,...,m$, completing the proof.
\end{proof}

\bibliographystyle{plain}
\bibliography{biblio}
\end{document}